\documentclass[12pt]{article}

\usepackage{amsmath,amssymb,amsthm,upref}
\usepackage{secdot}
\usepackage{geometry} 
\usepackage[ansinew]{inputenc}
\usepackage{dsfont}
\usepackage{mathrsfs}
\usepackage[all]{xy}
\usepackage{paralist}
\usepackage{graphicx}

\newcommand{\Dl}{\ensuremath{\Delta} }

\newcommand{\mc}[1]{\mathcal{#1}}
\newcommand{\mrm}[1]{\mathrm{#1}}
\newcommand{\mbf}[1]{\mathbf{#1}}

\newcommand{\mbb}[1]{\mathbb{#1}}
\newcommand{\comp}{{\scriptscriptstyle \ensuremath{\stackrel{\circ}{\mbox{}}}}}
\newcommand{\compc}{{\scriptscriptstyle \ensuremath{\stackrel{\star}{\mbox{}}}}}

\theoremstyle{theorem}
\newtheorem{thm}{\textsc{Theorem}}[section]

\newtheorem{prop}[thm]{\textsc{Proposition}}
\newtheorem{cor}[thm]{\textsc{Corollary}}

\theoremstyle{definition}
\newtheorem{defi}[thm]{\textsc{Definition}}

\newtheorem{obs}[thm]{\textbf{Remark}}

\begin{document}

\title{A derivability criterion based on the existence of adjunctions}

\author{Beatriz Rodr\'iguez Gonz\'alez\footnote{Partially supported by the ERC Starting Grant grant TGASS and the grants MTM2007-67908-C02-02, FQM-218 and SGR-119. 
The author gratefully acknowledges STORM department of London Metropolitan University 
for excellent working conditions.}\mbox{ }\footnote{\textit{email:} rgbea@icmat.es}\\
\begin{footnotesize}\textsc{Icmat-Csic-Complutense-Aut\'onoma-Carlos III}\end{footnotesize} }

\date{\empty}

\maketitle

\begin{abstract}
In this paper we introduce a derivability criterion of functors based on the existence of adjunctions rather than on the existence of resolutions. It constitutes 
a converse of Quillen-Maltsiniotis Derived Adjunction Theorem. We present two consequences of our derivability criterion. On the 
one hand, we prove that the two 
notions for homotopy colimits corresponding to Grothendieck derivators and Quillen model categories are equivalent.
On the other hand, we deduce that the internal hom for derived Morita theory constructed by B. To$\ddot{\mrm{e}}$n is indeed the right 
derived functor of the internal hom of dg-categories.
\begin{flushleft}\begin{tabular}{r}
\begin{small}{\textit{MSC:} 18A40, 18G10; 55U35, 18D15}\end{small}\\
\begin{small}{keywords: Homotopical Algebra; homotopy colimit; derived Morita theory} \end{small}
\end{tabular}
\end{flushleft}
\end{abstract}

\section{Introduction.}

In homological algebra, the classical criteria of existence of derived functors are always based on the existence of resolutions. 
For instance, it is a basic fact 
that if $\mc{A}$ is an abelian category with enough injectives then there exists the right derived functor $\mbb{R}F:\mc{D}^+(\mc{A})\rightarrow \mc{D}^+(\mc{B})$ 
of any left exact functor $F:\mc{A}\rightarrow\mc{B}$. More recently, the existence of the right derived functor 
$\mbb{R}F:\mc{D}(\mc{A})\rightarrow \mc{D}(\mc{B})$ of $F$ for unbounded complexes over a Grothendieck abelian category $\mc{A}$ is deduced in 
\cite{AJS} and \cite{S} from the existence of K-injective resolutions.
In the context of model categories, f{i}brant replacements ensure the existence of the right derived functor of any right Quillen functor (\cite{Q}). 
And there are other kinds of resolutions, for instance the right deformations of \cite{DHKS}, and the f{i}brant models of \cite{GNPR}.
A general notion of resolution that includes the preceding examples consists of the `structures de d{\'e}rivabilit{\'e}' developed in \cite{KM}, where 
a derivability criterion based on their existence is given.\\[0.2cm]
\indent In this paper we obtain a derivability criterion that uses a dif{f}erent approach:
instead of assuming the existence of resolutions, we assume the existence of adjunctions. It is motivated by the Quillen Derived Adjunction Theorem 
(\cite{Q}), as generalized by G. Maltsiniotis (\cite{M}):\\[0.1cm]
\noindent\textsc{Derived Adjunction Theorem.}
\textit{Let $F:\mc{C}\rightleftarrows \mc{D}: G$ be a pair of adjoint functors. Assume there exists $\mathbb{L}F$, the absolute left derived functor of $F$, and $\mbb{R}G$, the absolute right derived functor of $G$. Then $(\mbb{L}F, \mathbb{R}G)$ is again a pair of adjoint functors.}\\[0.2cm]
In other words, the Derived Adjunction Theorem states that adjunctions are preserved by taking absolute derived functors. 
Our derivability criterion is a converse of this fact:\\[0.2cm]
\textbf{Theorem \ref{criterion}.}
\textit{Let $F:\mc{C}\rightleftarrows \mc{D}: G$ be a pair of adjoint functors. Assume there exists $\mbb{L}F$, the absolute left derived functor of $F$. If $\mbb{L}F$ admits a right adjoint $G'$ 
$$\mbb{L}F: \mc{C}[\mc{W}^{-1}]\rightleftarrows \mc{D}[\mc{W}^{-1}]: G'$$
then the absolute right derived functor $\mbb{R}G$ of $G$ exists, and it agrees with $G'$.}\\[0.2cm]
We deduce this theorem from a result concerning iterated Kan extensions due to E. Dubuc (\cite{D}). 
Dually, for an adjunction $F:\mc{C}\rightleftarrows \mc{D}: G$ in which $G$ is absolutely right derivable 
and $F'$ is left adjoint to $\mbb{R}G$, then $\mbb{L}F$ exists and agrees with $F'$.\\[0.2cm]
\indent In the context of algebraic topology, a direct consequence of the above result is that
the two notions of homotopy colimit due to Grothen\-di\-eck and Quillen agree. 
More precisely, we deduce in Proposition \ref{GQ} that 
a functor is a homotopy colimit as defined in the setting of Grothendieck derivators (i.e. a left adjoint of the localized constant diagram functor) if and only if
it is a homotopy colimit as defined in the setting of Quillen model categories (i.e. an absolute left derived functor of the 
colimit). More generally, homotopy left Kan extensions may be equivalently def{i}ned \'a la Grothendieck or \'a la Quillen. 
A Corollary of this fact is that homotopy left Kan extensions \'a la Quillen are always composable 
(see Corollary \ref{composers}).\\[0.2cm]
\indent To finish, we describe a second consequence of Theorem \ref{criterion} that concerns B. To$\ddot{\mrm{e}}$n's internal 
hom for derived Morita theory, introduced in \cite{T}. Derived Morita theory is developed in loc. 
cit. using a suitable homotopy theory of dg-categories. A technical difficulty to do this was to show that the homotopy category of dg-categories, $Ho(dgcat)$, which 
is easily seen to be symmetric monoidal, is indeed a \textit{closed} symmetric monoidal category. This is proved directly in 
\cite{T}, providing an explicit 
construction for the internal hom, $\mc{R}\underline{\mrm{Hom}}(-,-)$, in $Ho(dgcat)$. Since 
this construction is not made though resolutions, 
it is not clear whether $\mc{R}\underline{\mrm{Hom}}(-,-)$ is the right derived functor of $\underline{\mrm{Hom}}(-,-)$ or not (see \cite{Ta2}). 
We settle this question proving in Corollary \ref{Toen} that this is indeed the case.\\

The paper is organized as follows. In section \ref{secPreliminaries} we introduce the preliminaries on absolute Kan extensions and absolute derived functors needed 
later. The third section contains the derivability criterion, and a corollary regarding the composition of derived functors. 
In section \ref{secGQ} we deduce the equivalence between the Quillen and Grothendieck notions of homotopy colimit. Finally, in the last section we apply our derivability criterion to derived internal homs,
deducing that To$\ddot{\mrm{e}}$n's internal hom of derived Morita theory is indeed the absolute right derived functor of the internal hom of dg-categories.\\

\noindent {\sc Acknowledgements:} I am indebted to V. Navarro Aznar for pointing out to me a more general form of the initial derivability criterion.  
I would also like to thank F. Muro, F. Guill{\'e}n Santos and A. Roig Maranges for useful discussions and comments. 
Finally, I gratefully acknowledge an anonymous referee on an earlier version of the manuscript 
for suggesting an alternative proof of the criterion based on mates theory.


\section{Preliminaries.}\label{secPreliminaries}

In this section we introduce some preliminaries regarding absolute Kan extensions and derived functors. We assume the reader is familiar with the basics on these 
topics, that may be found at \cite{ML}, or sections 2 and 3 of \cite{KM}, for instance.

\subsection{Absolute Kan extensions.}

Given functors $T: \mc{M}\rightarrow \mc{A}$ and $K:\mc{M}\rightarrow \mc{C}$, we denote by $(Ran_K T,\epsilon)$ the right Kan extension of $T$ along $K$ (if it exists). It consists of a functor $Ran_K T : \mc{C}\rightarrow\mc{A}$ and a natural transformation $\epsilon: Ran_K T \comp K\rightarrow T$, called the \textit{unit} of the Kan extension 
$$\xymatrix@M=4pt@H=4pt@C=35pt{
 \mc{M} \ar[r]^{T} \ar[d]_{K}  & \mc{A} \\
 \mc{C} \ar[ru]_{Ran_K T}^{\epsilon\,\,\mbox{\rotatebox[origin=c]{90}{$\Rightarrow$}}} & {}
}$$
that satisfy the usual universal property.
Recall that a right Kan extension $(Ran_K T,\epsilon)$ is characterized by the existence for each $L:\mc{C}\rightarrow\mc{A}$ of a natural bijection
$$\tau: \mrm{Nat}(L,Ran_K T)\longleftrightarrow \mrm{Nat}(L\comp K, T)$$
where $\mrm{Nat}(F,G)=\{$ natural transformations from $F$ to $G$ $\}$.
Indeed, given $(Ran_K T,\epsilon)$, $\tau$ is def{i}ned as the map that sends $\lambda: L\rightarrow Ran_K T$ to
$\epsilon\comp (\lambda\compc K): L\comp K\rightarrow T$. Conversely, given $\tau$ one takes $L=Ran_K T$, 
 $\epsilon=\tau(1_{Ran_K T}:Ran_K T\rightarrow Ran_K T)$.
A left Kan extension of $T$ along $K$ is def{i}ned dually by the existence of natural bijections
$$ \mrm{Nat}(Lan_K  T, L)\longleftrightarrow \mrm{Nat}(T, L\comp K)$$

\begin{defi}
A right Kan extension $(Ran_K T,\epsilon)$ is said to be \textit{absolute} if it is preserved by any functor. More concretely, given $S:\mc{A}\rightarrow \mc{B}$ then the right Kan extension of $S\comp T$ along $K$ exists and $Ran_K (S\comp T) = S \comp Ran_K T$. In addition, the unit $Ran_K(S\comp T)\comp K\rightarrow S\comp T$ is required to agree with $ S \compc \epsilon$. In this case, for each $L:\mc{C}\rightarrow\mc{A}$ there are bijections
$$ \mrm{Nat}(L, S \comp Ran_K T)\longleftrightarrow \mrm{Nat}(L\comp K, S\comp T)$$
natural on $S$ and $L$.
An \textit{absolute left Kan extension} of $T$ along $K$ is def{i}ned dually, so there are natural bijections
$ \mrm{Nat}(S\comp Lan_K  T, L)\longleftrightarrow \mrm{Nat}(S\comp T, L\comp K)$.
\end{defi}

Many of the constructions occurring in category theory may be expressed in terms of Kan extensions. An example is the case of adjunctions.

\begin{prop}\label{AdjKan} Given a functor $F:\mc{C}\rightarrow \mc{D}$, the following are equivalent:
\begin{compactitem}
\item[1.] $F$ admits a right adjoint $G:\mc{D}\rightarrow \mc{C}$.
\item[2.] There exists an absolute left Kan extension $G=Lan_F 1_{\mc{C}}$.
\item[3.] There exists a left Kan extension $G=Lan_F 1_{\mc{C}}$ and it is preserved by $F$.
\end{compactitem}
If these equivalent conditions hold, the units can be chosen in such a way that $\epsilon: 1_{\mc{C}}\rightarrow Lan_F 1_{\mc{C}}\comp F$ is both the unit of the adjunction and of the left Kan extension.
\end{prop}

\begin{proof} The equivalence between 1 and 3 is \cite[Theorem X.7.2]{ML}. On the other hand 2 implies 3 is obvious, while 1 implies 2 is \cite[Proposition X.7.3]{ML}.
\end{proof}

The dual result states in particular that $G$ admits a left adjoint $F$ if and only if there exists an absolute right Kan extension $F=Ran_{G}1_{\mc{D}}$ of $1_{\mc{D}}:\mc{D}\rightarrow\mc{D}$ along $G$.\\

We will use the following result about iterated Kan extensions due to E. Dubuc. Its proof may be found in \cite[Proposition I.4.1]{D}

\begin{prop}\label{Dubuc}
  Consider functors $K:\mc{M}\rightarrow\mc{C}$, $L:\mc{C}\rightarrow\mc{D}$ and
  $T:\mc{M}\rightarrow\mc{A}$. Assume there exists $Ran_K T$, the right Kan extension of $T$ along $K$. Then there exists $Ran_{L\comp K} T$ if and only if there exists $Ran_L Ran_K T$. In this case, both agree and the units may be chosen in such a way that
  $$\epsilon_{Ran_{L\comp K}T} = \epsilon_{Ran_K T}\comp (\epsilon_{Ran_L Ran_K T}\compc K  )$$
  In addition, if $Ran_K T$ is an absolute right Kan extension, then $Ran_L Ran_K T$ is absolute if and only if $Ran_{L\comp K} T$ is.
\end{prop}
\subsection{Absolute derived functors.}


Given a class $\mc{W}$ of morphisms in a category $\mc{C}$, the \textit{localization} of $\mc{C}$ with respect to $\mc{W}$ is the result of formally inverting the
morphisms of $\mc{W}$ in $\mc{C}$. This gives a (possibly big category) $\mc{C}[\mc{W}^{-1}]$ plus a localization functor
$\gamma_{\mc{C}}:\mc{C}\rightarrow \mc{C}[\mc{W}^{-1}]$ sending the morphisms of $\mc{W}$ to isomorphisms, and inducing for each category $\mc{D}$ an equivalence
of categories
$$- \compc \gamma_{\mc{C}}: Cat(\mc{C}[\mc{W}^{-1}],\mc{D}) \longrightarrow Cat_{\mc{W}}(\mc{C},\mc{D})$$
Here $Cat(\mc{C}[\mc{W}^{-1}],\mc{D})$ is the category of functors $G':\mc{C}[\mc{W}^{-1}]\rightarrow\mc{E}$ and $Cat_{\mc{W}}(\mc{C},\mc{D})$ is the full 
subcategory of $Cat(\mc{C},\mc{D})$ formed by those $G:\mc{C}\rightarrow\mc{E}$ that send the morphisms in $\mc{W}$ to isomorphisms.\\

If $F:\mc{C}\rightarrow\mc{D}$, the \textit{left derived functor} of $F$ (if it exists), is the right Kan
extension $(\mathbb{L}F=Ran_{\gamma_{\mc{C}}}F, \epsilon: \mbb{L}F\comp \gamma_{\mc{C}}\rightarrow F)$.
In case $\mathbb{L}F$ is in addition an absolute right Kan extension of $F$ along $\gamma_{\mc{C}}$, then $\mbb{L}F$ 
is called the \textit{absolute} left derived functor of $F$. In particular, for each $S:\mc{D}\rightarrow\mc{E}$,
$\mbb{L}(S\comp F) = S\comp \mbb{L}F$.\\

If $\mc{D}$ is also equipped with a distinguished class of morphisms, which we also write as $\mc{W}$,  then the (\textit{absolute}) \textit{total left derived 
functor} of $F:\mc{C}\rightarrow\mc{D}$ with respect to the classes $\mc{W}$ of $\mc{C}$ and $\mc{D}$
is the (absolute) right Kan extension of $\gamma_{\mc{D}}\comp F$ along $\gamma_{\mc{C}}$. It is also denoted by $\mbb{L}F$, and 
this time 
$$\mbb{L}F = Ran_{\gamma_{\mc{C}}}\gamma_{\mc{D}} \comp F$$
In the absolute case, for each functor $S:\mc{D}[\mc{W}^{-1}]\rightarrow \mc{E}$ it holds that  
$$S\comp \mbb{L}F = Ran_{\gamma_{\mc{C}}}S\comp \gamma_{\mc{D}} \comp F$$ 

In what follows we only consider the total case. Then, for the sake of brevity we will drop the `total' adjective and 
just say that $F$ has an (absolute) left derived functor $\mbb{L}F:\mc{C}[\mc{W}^{-1}]\rightarrow \mc{D}[\mc{W}^{-1}]$.\\
Dually, $F$ is said to admit an (\textit{absolute}) \textit{right derived functor} if there exists the (absolute) left Kan
extension $(\mathbb{R}F=Lan_{\gamma_{\mc{C}}}\gamma_{\mc{D}}\comp F, \epsilon: \gamma_{\mc{D}}\comp F \rightarrow \mbb{R}F\comp \gamma_{\mc{C}} )$.\\

From now on we f{i}x a class $\mc{W}$ of morphisms, called weak equivalences, in the categories considered, and left or right derived 
functors are always def{i}ned with respect to these classes.\\

Derived functors are typically obtained in practice through the existence of some kind 
of resolutions in $\mc{C}$. A general example is Quillen's Theorem of Existence of Derived Functors, in the context of model 
categories.

\begin{thm}\emph{(\cite{Q})}\label{QAT}
Let $F:\mc{M}\rightarrow \mc{D}$  be a functor from a Quillen model category $\mc{M}$ to a category $\mc{D}$. 
If $F$ sends weak equivalences between cof{i}brant objects to weak equivalences, then 
the left derived functor of $F$ exists and may be computed composing $F$ with a cof{i}brant replacement. Dually, if 
$F$ sends weak equivalences between f{i}brant objects to weak equivalences,
then the right derived functor of $F$ exists and may be computed composing $F$ with a f{i}brant replacement.
\end{thm}

\begin{obs}\label{QAT2} Although no mention to absoluteness is made in the original form of the above theorem, 
those derived functors obtained through cof{i}brant or f{i}brant resolutions are easily seen to be absolute derived functors 
(see \cite[p.2]{M}).
\end{obs}

 Previous theorem shows that a suf{f}{i}cient condition on a functor def{i}ned on a model category for being absolutely right derivable is that it preserves 
weak equivalences between f{i}brant objects. Indeed, some authors directly def{i}ne the (absolute) right derived functor of such a functor as its composition 
with a f{i}brant replacement, instead as its (absolute) right Kan extension. Both def{i}nitions agree in this case by the above theorem.\\ 

But note however that this is not a necessary condition on a functor to be absolutely right derivable: 
just consider a model category in which \textit{all} objects are fibrant, 
for instance the projective model structure on complexes of abelian groups.
With this example in mind it is clear that there are many functors that do not preserve weak equivalences between 
f{i}brant (i.e. all) objects but still admit an absolute right derived functor. 
Another example is the internal hom of dg-categories (see section \ref{dM}).


\section{Derivability criterion.}

In this section we introduce the derivability criterion (Theorem \ref{criterionLF}), and study the composition of derived functors in the presence of adjoints. 
We begin by recalling Quillen's Derived Adjunction Theorem (\cite{Q}), as generalized by G. Maltsiniotis (\cite{M}):\\[0.2cm]
\textbf{Derived Adjunction Theorem.}
\textit{Let $F:\mc{C}\rightleftarrows \mc{D}: G$ be a pair of adjoint functors. Assume there exists $\mathbb{L}F$, the absolute left derived functor of $F$, and $\mbb{R}G$, the absolute right derived functor of $G$. Then
$$\mbb{L}F: \mc{C}[\mc{W}^{-1}]\rightleftarrows \mc{D}[\mc{W}^{-1}]: \mathbb{R}G$$
is a pair of adjoint functors. In addition, the adjunction morphisms $a':\mbb{L}F\comp \mathbb{R}G \rightarrow 1$ and $b':1\rightarrow \mathbb{R}G\comp\mbb{L}F$ may be chosen in such a way
that the two squares below commute
\begin{equation}\label{comptMorf}\xymatrix@M=4pt@H=4pt@C=35pt{
 \mbb{L}F\comp\gamma_{\mc{C}} \comp G \ar[r]^-{\mbb{L}F\compc \varepsilon} \ar[d]_-{\epsilon\compc G}  & \mbb{L}F\comp \mbb{R}G\comp \gamma_{\mc{D}} \ar[d]^-{a'\compc\gamma_{\mc{D}}}  & & \mbb{R}G\comp\gamma_{\mc{D}}\comp F  & \mbb{R}G\comp\mbb{L}F\comp\gamma_{\mc{C}}\ar[l]_-{\mbb{R}G\compc \epsilon} \\
 \gamma_{\mc{D}}\comp F\comp G \ar[r]_-{\gamma_{\mc{D}}\compc a} & {\gamma_{\mc{D}}} & & \gamma_{\mc{C}}\comp G\comp F \ar[u]^{\varepsilon\compc F} & \ar[l]_-{\gamma_{\mc{C}}\compc b} \gamma_{\mc{C}} \ar[u]_{b'\compc\gamma_{\mc{C}}}
}\end{equation}
Here, $\epsilon$ and $\varepsilon$ denote the respective units of $\mbb{L}F$ and $\mbb{R}G$, while $a$ and $b$ are the adjunction morphisms of $(F,G)$.}\\[0.2cm]
The following converse of previous theorem constitutes a derivability criterion of functors based on the existence 
of adjunctions rather than on the existence of resolutions.

\begin{thm}\label{criterionLF} Let $F:\mc{C}\rightleftarrows \mc{D}: G$ be a pair of adjoint functors. Assume there exists $\mbb{R}G$, the absolute right derived functor of $G$. If $\mbb{R}G$ admits a left adjoint $F'$
$$F': \mc{C}[\mc{W}^{-1}]\rightleftarrows \mc{D}[\mc{W}^{-1}]: \mathbb{R}G$$
then the absolute left derived functor $\mbb{L}F$ of $F$ exists, and it agrees with $F'$. In addition, the unit $\epsilon:F'\comp\gamma_{\mc{C}}\rightarrow\gamma_{\mc{D}}\comp F$ may be chosen in such a way that the two squares in \emph{(\ref{comptMorf})} commute.
\end{thm}

\begin{proof}
We must prove that $F'$ is the absolute right Kan extension $Ran_{\gamma_{\mc{C}}}\gamma_{\mc{D}}\comp F$ of $\gamma_{\mc{D}}\comp F$ along $\gamma_{\mc{C}}$. 
Consider $\mc{S}:\mc{D}[\mc{W}^{-1}]\rightarrow \mc{E}$. Since $(F',\mbb{R}G)$ is a pair of adjoint functors, it follows from Proposition \ref{AdjKan} that $S\comp F' = Ran_{\mbb{R}G} S$ with unit $S\compc a' : S\comp F'\comp \mbb{R}G \rightarrow S$. On the other hand,
by assumption $\mbb{R}G$ is the absolute right derived functor of $G$. Therefore, given $L:\mc{C}[\mc{W}^{-1}]\rightarrow \mc{E}$
then $L\comp \mbb{R}G = Lan_{\gamma_{\mc{D}}} L\comp \gamma_{\mc{C}}\comp G$ with unit $L\compc \varepsilon : L\comp\gamma_{\mc{C}}\comp G \rightarrow  L\comp \mbb{R}G \comp \gamma_{\mc{D}}$. Putting all together, we have the bijections
$$\begin{array}{ccccc}
 Nat(L,S\comp F') & \longrightarrow & Nat(L\comp \mbb{R}G, S) & \longrightarrow & Nat(L\comp \gamma_{\mc{C}}\comp G, S\comp \gamma_{\mc{D}})\\[0.1cm]
 \lambda & \mapsto & (S\compc a')\comp (\lambda\compc \mbb{R}G) & \mapsto  & (S\compc a' \compc \gamma_{\mc{D}})\comp (\lambda\compc (\mbb{R}G\comp \gamma_{\mc{D}}))\comp (L\compc\varepsilon)=\\
         &         &                                            &          & = (S\compc a' \compc \gamma_{\mc{D}})\comp  ((S\comp F') \compc \varepsilon )) \comp (\lambda\compc (\gamma_{\mc{C}}\comp G))                            
 \end{array}
$$
Then, we deduce that $F'$ is the absolute right Kan extension $Ran_{\gamma_{\mc{C}}\comp G}\gamma_{\mc{D}}$ with unit the composition
$$\xymatrix@M=4pt@H=4pt@C=40pt{ F'\comp \gamma_{\mc{C}} \comp G \ar[r]^-{ F'\compc \varepsilon} & F' \comp \mbb{R}G\comp\gamma_{\mc{D}} \ar[r]^-{ a'\compc \gamma_{\mc{D}}} & \gamma_{\mc{D}} }$$
On the other hand, since $(F,G)$ is a pair of adjoint functors then $F$ is the absolute right Kan extension $Ran_{G} 1_{\mc{D}}$ with unit 
$a : F\comp G\rightarrow 1_{\mc{D}}$, by Proposition \ref{AdjKan}. In particular, $\gamma_{\mc{D}}\comp F$ is the absolute 
right Kan extension $Ran_{G}\gamma_{\mc{D}}$ with unit
$\gamma_{\mc{D}}\compc a : \gamma_{\mc{D}}\comp F\comp G\rightarrow \gamma_{\mc{D}}$.
Summarizing all we have proved the existence of
$$Ran_{G} \gamma_{\mc{D}}= \gamma_{\mc{D}}\comp F \ \ \mbox{ and }\ \ Ran_{\gamma_{\mc{C}}\comp G} \gamma_{\mc{D}}=  F'$$
By Dubuc's result \ref{Dubuc} we deduce that $Ran_{\gamma_{\mc{C}}}Ran_{G} \gamma_{\mc{D}}= Ran_{\gamma_{\mc{C}}} \gamma_{\mc{D}}\comp F$ exists, it is absolute, and it agrees with $ F'$. In addition, we may choose the unit 
$\epsilon: (Ran_{\gamma_{\mc{C}}}\gamma_{D}\comp F)\comp \gamma_{C}\rightarrow \gamma_{\mc{C}}\comp F$ in such a way such that
$$\xymatrix@M=4pt@H=4pt@C=40pt{ F' \comp \gamma_{\mc{C}} \comp G \ar[r]^{\epsilon\compc G}\ar[rd]_{(a'\compc \gamma_{\mc{D}})\comp (F'\compc \varepsilon)\,\,\,\,} & \gamma_{\mc{D}}\comp F\comp G \ar[d]^{\gamma_{\mc{D}}\compc a} \\
 & \gamma_{\mc{D}} }$$
commutes, so the left square in (\ref{comptMorf}) is commutative. This formally implies
the commutativity of the right square too. Indeed, this  means that the unit 
$\epsilon: F'\comp\gamma_{\mc{C}} \rightarrow \gamma_{D}\comp F$ of the absolute left Kan extension is the
adjoint natural transformation through $({F'},\mbb{R}G)$ of 
\begin{equation}\label{unit}\xymatrix@M=4pt@H=4pt@C=33pt{\gamma_{\mc{C}}\ar[r]^-{\gamma_{\mc{C}} \compc b} & \gamma_{\mc{C}}\comp G\comp F 
\ar[r]^-{\varepsilon\compc F} & 
\mbb{R}G\comp \gamma_{\mc{D}}\comp F } 
\end{equation}
Using the naturality of $\epsilon$ and the left square in (\ref{comptMorf}) we deduce the commutative diagram 
$$\xymatrix@M=4pt@H=4pt@C=38pt{ F'\comp \gamma_{\mc{C}}\ar[r]^-{(F'\comp\gamma_{\mc{C}}) \compc b} \ar[rd]_{\epsilon} & F'\comp\gamma_{\mc{C}}\comp G\comp F 
\ar[r]^-{F'\compc\varepsilon\compc F} \ar[rd]^{\epsilon\compc (G\comp F)} & 
F'\comp\mbb{R}G\comp \gamma_{\mc{D}}\comp F  \ar[r]^-{a'\compc (\gamma_{D}\comp F)} &  \gamma_{\mc{D}}\comp F\\
& \gamma_{\mc{D}}\comp F \ar[r]^-{(\gamma_{\mc{D}}\comp F)\compc b } & \gamma_{\mc{D}}\comp F\comp G\comp F \ar[ru]_{ \gamma_{\mc{D}}\compc a \compc F} & }$$
Since $(a\compc F)\comp (F\compc b)$ is the identity, it follows that the adjoint of (\ref{unit}) is $\epsilon$ as required.
\end{proof}

\begin{obs} An alternative proof of previous theorem may be obtained using the formalism of `mates'. More concretely, 
there are two double categories whose objects are categories and whose vertical morphisms are, respectively, right and 
left adjoints. They are related through the natural morphism that sends a right adjoint to its left adjoint and 
a vertical morphism to its \textit{mate}. We left as an exercise for the interested reader to deduce previous theorem from the 
fact that this morphism is indeed an isomorphism (\cite[Proposition 2.2]{KS}). In this setting, the commutativity of each of the 
two squares in (1) (which are indeed equivalent) just means that $\varepsilon$ and $\epsilon$ are mates.
\end{obs}

Putting together the previous two results, we obtain the following characterization of left derivability for a functor
possessing an absolutely right derivable right adjoint. 

\begin{cor}\label{characterizationLF} Let $F:\mc{C}\rightleftarrows \mc{D}: G$ be a pair of adjoint functors. Assume there exists $\mbb{R}G$, the absolute right derived functor of $G$. Then, the following are equivalent:
\begin{compactitem}
\item[1.] $\mbb{R}G$ admits a left adjoint
$F': \mc{C}[\mc{W}^{-1}]\rightarrow \mc{D}[\mc{W}^{-1}]$.
\item[2.] $F$ admits an absolute left derived functor $\mbb{L}F: \mc{C}[\mc{W}^{-1}]\rightarrow \mc{D}[\mc{W}^{-1}]$.
\end{compactitem}
In this case, $F'$ and $\mbb{L}F$ agree and their corresponding unit and adjunction morphisms may be chosen in such a way that the two squares in \emph{(\ref{comptMorf})} commute.
\end{cor}

We will also use the duals of Theorem \ref{criterionLF} and Corollary \ref{characterizationLF}, given below.

\begin{thm}\label{criterion}
Let $F:\mc{C}\rightleftarrows \mc{D}: G$ be a pair of adjoint functors. Assume there exists $\mbb{L}F$, the absolute left derived functor of $F$. If $\mbb{L}F$ admits a right adjoint $G'$ 
$$\mbb{L}F: \mc{C}[\mc{W}^{-1}]\rightleftarrows \mc{D}[\mc{W}^{-1}]: G'$$
then the absolute right derived functor $\mbb{R}G$ of $G$ exists, and it agrees with $G'$. In addition, the unit $\varepsilon:\gamma_{\mc{C}}\comp G \rightarrow G'\comp\gamma_{\mc{D}}$ may be chosen in such a way that the two squares in \emph{(\ref{comptMorf})} commute.
\end{thm}

\begin{cor}\label{characterization} Let $F:\mc{C}\rightleftarrows \mc{D}: G$ be a pair of adjoint functors. Assume there exists $\mbb{L}F$, the absolute left derived functor of $F$. Then, the following are equivalent:
\begin{compactitem}
\item[1.] $\mbb{L}F$ admits a right adjoint
$G': \mc{D}[\mc{W}^{-1}]\rightarrow \mc{C}[\mc{W}^{-1}]$.
\item[2.] $G$ admits an absolute right derived functor $\mbb{R}G: \mc{D}[\mc{W}^{-1}]\rightarrow \mc{C}[\mc{W}^{-1}]$.
\end{compactitem}
In this case, $G'$ and $\mbb{R}G$ agree and their corresponding unit and adjunction morphisms may be chosen in such a way that the two squares in \emph{(\ref{comptMorf})} commute.
\end{cor}

We now turn to the composition of derived functors obtained in this way. Given functors $F_1:\mc{C}\rightarrow \mc{D}$ and $F_2:\mc{D}\rightarrow \mc{E}$
such that $\mbb{L}F_1$, $\mbb{L}F_2$ and $\mbb{L}(F_2\comp F_1)$ exist, there is a natural morphism
\begin{equation}\label{compL} \mbb{L}F_2\comp \mbb{L}F_1\longrightarrow \mbb{L}(F_2\comp F_1) \end{equation}
obtained from $(\epsilon_{\mbb{L}F_2}\compc F_1)\comp (\mbb{L}F_2\compc\epsilon_{\mbb{L}F_1}) :\mbb{L}F_2\comp\mbb{L}F_1\comp \gamma_{\mc{C}}\rightarrow \gamma_{\mc{E}}\comp F_2\comp F_1$ through the bijection
$$Nat(\mbb{L}F_2\comp\mbb{L}F_1,\mbb{L}(F_2\comp F_1))\longrightarrow  Nat(\mbb{L}F_2\comp\mbb{L}F_1\comp \gamma_{\mc{C}},\gamma_{\mc{E}}\comp F_2\comp F_1)$$

Analogously, given $G_1:\mc{D}\rightarrow \mc{C}$ and $G_2:\mc{E}\rightarrow \mc{D}$ such that there exist $\mbb{R}G_1$, $\mbb{R}G_2$ and $\mbb{R}(G_1 \comp G_2)$, there is a natural morphism
\begin{equation}\label{compR} \mbb{R}(G_1\comp G_2)\longrightarrow \mbb{R}G_1\comp \mbb{R}G_2
\end{equation}

Note that neither (\ref{compL}) nor (\ref{compR}) need to be an isomorphism for general $F_1$, $F_2$, $G_1$ and $G_2$. In our case we have the

\begin{prop}\label{composicionDerivados}
Consider two pairs of adjoint functors $(F_1,G_1)$ and $(F_2,G_2)$
$$\xymatrix@M=4pt@H=4pt@C=33pt{ \mc{C} \ar@<0.5ex>[r]^-{F_{1}}  & \mc{D} \ar@<0.5ex>[l]^-{G_1} \ar@<0.5ex>[r]^-{F_{2}}  & \mc{E} \ar@<0.5ex>[l]^-{G_2} }$$
such that the absolute derived functors $\mbb{L}F_1$, $\mbb{L}F_2$,
$\mbb{R}G_1$, $\mbb{R}G_2$ and $\mbb{R}(G_1\comp G_2)$ exist. If \emph{(\ref{compR})} is an isomorphism, then there exists the absolute left derived functor $\mbb{L}(F_2\comp F_1)$ and \emph{(\ref{compL})} is an isomorphism as well.
\end{prop}

\begin{proof}
By the Derived Adjunction Theorem we have the adjoint pairs of functors
$(\mbb{L}F_1, \mbb{R}G_1)$ and $(\mbb{L}F_2, \mbb{R}G_2)$.
Since the composition of adjunctions is again an adjunction, we have also the adjoint pair $(\mbb{L}F_2\comp\mbb{L}F_1, \mbb{R}G_1\comp\mbb{R}G_2)$.
If (\ref{compR}) is an isomorphism, it turns out that $\mbb{L}F_2\comp\mbb{L}F_1$
is left adjoint to $\mbb{R}(G_1\comp G_2)$. But then, by Corollary \ref{characterizationLF} applied to $(F_2\comp F_1,G_1\comp G_2)$ we have that 
$\mbb{L}F_2\comp\mbb{L}F_1$ is the absolute left derived functor of $F_2\comp F_1$.
\end{proof}

It is left to the reader to establish the dual of previous proposition, concerning the composition
of absolute right derived functors. 


\section{Equivalence between Grothendieck and Quillen homotopy colimits.}\label{secGQ}

A nice particular case of Corollary \ref{characterizationLF} occurs when in an adjunction $F:\mc{C}\rightleftarrows \mc{D}: G$ the right adjoint preserves weak equivalences. In this case
we deduce that $F$ admits an absolute left derived functor if and only if
$G:\mc{D}[\mc{W}^{-1}]\rightarrow \mc{C}[\mc{W}^{-1}]$ admits a left adjoint, and in this case both agree.
In the setting of homotopy colimits, a consequence of this fact is that
the notions of homotopy colimit corresponding to Grothendieck derivators and
model categories are indeed equivalent. We begin by recalling how these two notions are def{i}ned.\\

Given a small category $I$, the class $\mc{W}$ of weak equivalences of $\mc{C}$ induces object-wise a class of morphisms in the 
category of functors from $I$ to $\mc{C}$, $\mc{C}^I$, called the class of \textit{pointwise weak equivalences} and denoted also by $\mc{W}$.   
More concretely, a pointwise weak equivalence of $\mc{C}^I$ is a natural transformation $\lambda:X\rightarrow Y$ such that $\lambda_i\in\mc{W}$ for each $i\in I$.
Note that the constant diagram functor $c_I:\mc{C}\rightarrow\mc{C}^I$, def{i}ned as $(c_I(x))(i)=x$ for all $i\in I$, is then weak equivalence preserving. 
We also denote by $c_I$ the induced functor ${c}_I:\mc{C}[\mc{W}^{-1}]\rightarrow \mc{C}^I[\mc{W}^{-1}] $ on localized categories, and call it the 
\textit{localized constant diagram functor}.

\begin{defi} A \textit{Grothendieck homotopy colimit} is def{i}ned as the left 
adjoint $\mrm{hocolim}_I:\mc{C}^I[\mc{W}^{-1}]\rightarrow \mc{C}[\mc{W}^{-1}]$ of the localized constant diagram functor 
$c_I:\mc{C}[\mc{W}^{-1}]\rightarrow \mc{C}^I[\mc{W}^{-1}]$ (if it exists).\\
On the other hand, if there exists the colimit $\mrm{colim}_I:\mc{C}^I\rightarrow \mc{C}$, a \textit{Quillen homotopy colimit} is def{i}ned as the 
absolute left derived functor $\mbb{L}\mrm{colim}_I:\mc{C}^I[\mc{W}^{-1}]\rightarrow \mc{C}[\mc{W}^{-1}]$ of the colimit 
$\mrm{colim}_I$ (if it exists).
\end{defi}

\begin{prop}\label{GQ} Assume there exists $\mrm{colim}_I:\mc{C}^I\rightarrow \mc{C}$. 
Then a functor $\mc{C}^I[\mc{W}^{-1}]\rightarrow \mc{C}[\mc{W}^{-1}]$ is a Grothendieck homotopy colimit if and only if it is a Quillen homotopy colimit.
\end{prop}

\begin{proof} Since by hypothesis $\mrm{colim}_I:\mc{C}^I\rightarrow \mc{C}$ exists, we have an adjunction 
$\mrm{colim}_I:\mc{C}^I\rightleftarrows \mc{C}: c_I$. Since $c_I$ is weak equivalence preserving, it has in particular
an absolute right derived functor and $\mbb{R}c_I = c_I : \mc{C}^I[\mc{W}^{-1}]\rightarrow \mc{C}[\mc{W}^{-1}]$. Hence the result follows directly from Corollary \ref{characterizationLF}.
\end{proof}

Recall that the inverse image of a functor $f:I\rightarrow J$ of small categories is $f^\ast: \mc{C}^J\rightarrow \mc{C}^I$ def{i}ned as $(f^{\ast}Y)(i)=Y(f(i))$. 
As before, $f^\ast$ clearly preserves pointwise weak equivalences.

\begin{defi}
A \textit{Grothendieck homotopy left Kan extension} along $f$, $Ho f_!:\mc{C}^I[\mc{W}^{-1}]\rightarrow \mc{C}^J[\mc{W}^{-1}]$, is def{i}ned as the 
left adjoint of $f^\ast: \mc{C}^J[\mc{W}^{-1}]\rightarrow \mc{C}^I[\mc{W}^{-1}]$  (if it exists).\\ 
On the other hand, if there exists the left Kan extension $f_!: \mc{C}^I\rightarrow \mc{C}^J$, a \textit{Quillen homotopy left Kan extension} along $f$ is def{i}ned as 
the absolute left derived functor $\mbb{L} f_!:\mc{C}^I[\mc{W}^{-1}]\rightarrow \mc{C}^J[\mc{W}^{-1}]$ of $f_!$ (if it exists).
\end{defi}

Analogously, we deduce from Corollary \ref{characterizationLF} the 

\begin{prop} Assume there exists $f_!: \mc{C}^I\rightarrow \mc{C}^J$. Then a functor $\mc{C}^I[\mc{W}^{-1}]\rightarrow \mc{C}^J[\mc{W}^{-1}]$
is a Grothendieck homotopy left Kan extension if and only if it is a Quillen homotopy left Kan extension.
\end{prop}

In light of this equivalence, we deduce in next Corollary that Quillen homotopy left Kan extensions are always composable.

\begin{cor}\label{composers}
Assume given functors $f:I\rightarrow J$ and $g:J\rightarrow K$ such that there exist the absolute left derived functors 
$\mbb{L}f_{!}:\mc{C}^I[\mc{W}^{-1}]\rightarrow \mc{C}^J[\mc{W}^{-1}]$ and 
$\mbb{L}g_{!} : \mc{C}^J[\mc{W}^{-1}]\rightarrow \mc{C}^K[\mc{W}^{-1}]$. Then there exists the absolute left derived functor $\mbb{L}(g\comp f)_{!}$ and it agrees with $\mbb{L}g_!\comp\mbb{L}f_!$.
\end{cor}

\begin{proof} Since $f^\ast$, $g^\ast$ and $f^{\ast}\comp g^{\ast}= (g\comp f)^{\ast}$ are weak equivalence preserving functors, in particular $\mbb{R}(f^{\ast}\comp g^{\ast})=f^{\ast}\comp g^{\ast}= \mbb{R}f^{\ast}\comp \mbb{R}g^{\ast}$. Hence the corresponding composition morphism (\ref{compR}) is an isomorphism, and the result follows from Proposition \ref{composicionDerivados} applied to
$$\xymatrix@M=4pt@H=4pt@C=33pt{ \mc{C}^I \ar@<0.5ex>[r]^-{f_{!}}  & \mc{C}^J \ar@<0.5ex>[l]^-{f^{\ast}} \ar@<0.5ex>[r]^-{g_{!}}  & \mc{C}^K \ar@<0.5ex>[l]^-{g^{\ast}} }$$
\end{proof}


The dual results obtained using Corollary \ref{characterization} assert that, when both are defined, Gro\-then\-dieck homotopy limits (def{i}ned as right 
adjoints of the localized constant diagram functors) are equivalent to Quillen homotopy limits (def{i}ned as absolute right derived 
functors of limits). Also, Grothendieck 
homotopy right Kan extensions are equivalent to Quillen homotopy right Kan extensions, def{i}ned dually, and consequently Quillen homotopy right Kan extensions 
are always composable.

\begin{obs} Despite Gro\-then\-dieck and Quillen homotopy limits agree 
in case both are defined, that is, in case $\mc{C}$ is closed by limits, note that Grothendieck homotopy limits are
more general than Quillen homotopy limits. To see an example, consider the category $\mc{H}dg$ of 
 mixed Hodge complexes. By the results in \cite{R} it holds that Deligne's cosimplicial construction 
$\mbf{s}:\Dl\mc{H}dg\rightarrow \mc{H}dg$ is (after localizing by the weak equivalences) a Grothendieck homotopy limit. 

But $\mc{H}dg$ is not closed by limits of cosimplicial shape. 
The reason is that a mixed Hodge complex is given, among other things, by a couple of quasi-isomorphisms (relating its rational and complex components), 
and cosimplicial limits (which are in turn equalizers) do not preserve weak equivalences. 
So the Quillen homotopy limit of cosimplicial mixed Hodge complexes does not exist, but the 
Grothendieck homotopy limit $\mrm{holim}_{\Delta}$ does.   
\end{obs}


\section{Closed homotopy categories.}

Recall that a symmetric monoidal category $(\mc{C},\otimes:\mc{C}\times \mc{C}\rightarrow \mc{C})$ is said to be \textit{closed} if for each object $B$ of 
$\mc{C}$ the functor $-\otimes B:\mc{C}\rightarrow \mc{C}$ has a right adjoint $\underline{\mrm{Hom}}_{\mc{C}}(B,-)$. This means that there are natural bijections
$$ \mrm{Hom}_{\mc{C}}(A\otimes B , C) \longleftrightarrow  \mrm{Hom}_{\mc{C}}(A, \underline{\mrm{Hom}}_{\mc{C}}(B,C)) $$
In this case, $\underline{\mrm{Hom}}_{\mc{C}}(B,-)$ is also natural on $B$ producing a bifunctor 
$\underline{\mrm{Hom}}_{\mc{C}}(-,-):\mc{C}^{\comp}\times\mc{C}\rightarrow\mc{C}$.\\

If $\mc{W}$ is a class of morphisms in $\mc{C}$, a closed symmetric monoidal structure on $\mc{C}$ does not necessarily induce one on $\mc{C}[\mc{W}^{-1}]$. In case $\otimes$ passes to the localized category $\mc{C}[\mc{W}^{-1}]$, we deduce the following result.

\begin{prop}\label{DerHom} Let $(\mc{C},\otimes,\underline{\mrm{Hom}})$ be a closed symmetric monoidal category. Assume that
$(\mc{C}[\mc{W}^{-1}],\otimes^{\mbb{L}})$ is a symmetric monoidal category in which $-\otimes^{\mbb{L}}B$ is the absolute left 
derived functor of $-\otimes B$, for each object $B$ of $\mc{C}$. Then, the following are equivalent:
\begin{compactitem}
\item[1.] $(\mc{C}[\mc{W}^{-1}],\otimes^{\mbb{L}})$ is a \emph{closed} symmetric monoidal category, with internal hom 
$$\underline{\mrm{Hom}}_{\mc{C}[\mc{W}^{-1}]}(-,-): \mc{C}[\mc{W}^{-1}]^{\comp}\times\mc{C}[\mc{W}^{-1}]\rightarrow\mc{C}[\mc{W}^{-1}]$$
\item[2.] For each object $B$ of $\mc{C}$, the internal hom $\underline{\mrm{Hom}}_{\mc{C}}(B,-):\mc{C}\rightarrow \mc{C}$
has an absolute right derived functor $\mbb{R}\underline{\mrm{Hom}}_{\mc{C}}(B,-)$.
\end{compactitem}
In addition, if these equivalent conditions hold, then $\underline{\mrm{Hom}}_{\mc{C}[\mc{W}^{-1}]}(B,-)$ and 
$\mbb{R}\underline{\mrm{Hom}}_{\mc{C}}(B,-)$ agree for each object $B$ of $\mc{C}$. 
\end{prop} 

\begin{proof} For a f{i}xed object $B$ of $\mc{C}$, the left adjoint in the adjunction 
$-\otimes B:\mc{C}\rightleftarrows \mc{C}:\underline{\mrm{Hom}}_{\mc{C}}(B,-)$ admits by assumption an
absolute left derived functor $\mbb{L}(-\otimes B) = -\otimes^{\mbb{L}} B$. By Corollary \ref{characterization}, 
it follows that $-\otimes^{\mbb{L}} B: \mc{C}[\mc{W}^{-1}]\rightarrow\mc{C}[\mc{W}^{-1}]$ has a right adjoint
$\underline{\mrm{Hom}}_{\mc{C}[\mc{W}^{-1}]}(B,-)$ if and only $\underline{\mrm{Hom}}_{\mc{C}}(B,-)$ has an absolute right
derived functor $\mbb{R}\underline{\mrm{Hom}}_{\mc{C}}(B,-)$, and in this case they agree.
\end{proof}

\subsection{The internal hom of derived Morita theory.}\label{dM}

In case that $(\mc{C},\mc{W})$ is a \textit{closed symmetric monoidal model category}, $\mc{C}[\mc{W}^{-1}]$ does inherit a closed symmetric monoidal structure from $\mc{C}$. Indeed, it is given by the (absolute) left derived functor of $\otimes:\mc{C}\times \mc{C}\rightarrow \mc{C}$ and the (absolute) right derived functor of $\underline{\mrm{Hom}}:\mc{C}^{\comp}\times \mc{C}\rightarrow \mc{C}$ (see \cite[Theorem 4.3.2]{Ho}).

However, in some interesting situations one encounters a model category $(\mc{C},\mc{W})$ such that $\mc{C}$ is also closed symmetric monoidal, but the two 
structures are not compatible and consequently $(\mc{C},\mc{W})$ is not a symmetric monoidal model category. This is precisely the case of derived Morita theory, 
which is developed in \cite{T} in terms of dif{f}erential graded categories (or dg-categories for short).\\

Recall that a  dg-category $\mc{A}$ is a category enriched over the category $\mrm{C}(k)$ of complexes of $k$-modules, where $k$ is some f{i}xed ring.
From a dg-category $\mc{A}$ one can construct a usual category $[\mc{A}]$ with same objects as $\mc{A}$, and morphisms $\mrm{Hom}_{[\mc{A}]}(x,y) = \mrm{H}^0 \mrm{Hom}_{\mc{A}}(x,y)$.
 
As explained in \cite{T}, derived Morita theory may be established through a suitable homotopy theory of dg-categories. To do this, a \textit{weak equivalence} 
of dg-categories is def{i}ned as a dg-functor $f:\mc{A}\rightarrow \mc{B}$ such that $\mrm{Hom}(f)$ is a quasi-isomorphism of complexes, and $[f]$ is essentially 
surjective (hence an equivalence of categories). We write $Ho(dgcat)=dgcat[\mc{W}^{-1}]$.\\

With this notion of weak equivalences, $(dgcat,\mc{W})$ is indeed a model category (see \cite{Ta1}). However, with this model structure $
\otimes: dgcat\times dgcat\rightarrow dgcat$ does not preserve cof{i}brant objects, so $(dgcat,\mc{W})$ is not a closed symmetric monoidal model category.

The situation is not as bad though, since for a cof{i}brant dg-category $\mc{A}$ it holds that  $\mc{A}\otimes -: dgcat \rightarrow dgcat$ still preserves weak 
equivalences. This readily implies that for an arbitrary dg-category $\mc{B}$, $-\otimes \mc{B}$ preserves weak equivalences between cof{i}brant objects. 
Then, it follows from Theorem \ref{QAT} and Remark \ref{QAT2} the

\begin{prop} Given a dg-category $\mc{B}$,  $-\otimes \mc{B} : dgcat \rightarrow dgcat$ admits an absolute left derived functor  
$-\otimes^{\mbb{L}} \mc{B} : Ho(dgcat)  \rightarrow Ho(dgcat)$.
\end{prop}

With this derived tensor product $Ho(dgcat)$ becomes a \textit{closed} symmetric monoidal category (see \cite{T}). For each 
dg-category $\mc{B}$, the right adjoint $\mc{R}\underline{\mrm{Hom}}(\mc{B},-)$ to $-\otimes^{\mbb{L}} \mc{B}$ is explicitly  
constructing explicitly in loc. cit. as follows. First, one considers the dg-category 
$Int((\mc{B}\otimes^{\mathbb{L}}\mc{C}^{\comp})-mod)$ with objects the cof{i}brant $(\mc{B}\otimes^{\mathbb{L}}\mc{C})^{\comp}$-modules and
$$\mrm{Hom}_{Int((\mc{B}\otimes^{\mathbb{L}}\mc{C}^{\comp})-mod)}(E,F) = \underline{\mrm{Hom}}(E,F)$$
where the right hand side is the $\mrm{C}(k)$-valued hom of $(\mc{B}\otimes^{\mathbb{L}}\mc{C}^{\comp})-mod$.
Then, $\mc{R}\underline{\mrm{Hom}}(\mc{B},\mc{C})$ is the full sub-dg-category  $Int((\mc{B}\otimes^{\mathbb{L}}\mc{C}^{\comp})-mod)^{rqr}$ of 
$Int((\mc{B}\otimes^{\mathbb{L}}\mc{C}^{\comp})-mod)$ formed by right quasi-representable modules.\\

Since $\mc{R}\underline{\mrm{Hom}}(-,-)$ is not constructed through resolutions, it was not clear 
whether it was or not the right derived functor of $\underline{\mrm{Hom}}(\mc{B},-)$ (see \cite{Ta2}).
But a direct consequence of Proposition \ref{DerHom} is the

\begin{cor}\label{Toen} Given a dg-category $\mc{B}$, the internal hom $\mc{R}\underline{\mrm{Hom}}(\mc{B},-):Ho(dgcat)\rightarrow Ho(dgcat)$ of derived Morita 
theory is the absolute right derived functor $\mathbb{R}\underline{\mrm{Hom}}(\mc{B},-)$ of the internal hom 
$\underline{\mrm{Hom}}(\mc{B},-): dgcat\rightarrow dgcat$ of dg-categories.
\end{cor}

It also holds that $\mc{R}\underline{\mrm{Hom}}(-,-):Ho(dgcat)^{\comp}\times Ho(dgcat)\rightarrow Ho(dgcat)$ is the absolute 
right derived functor of $\underline{\mrm{Hom}}(-,-):dgcat^{\comp}\times dgcat\rightarrow dgcat$. 

This may be proved using previous result and the fact that the unit of the absolute right Kan extension
$-\otimes^{\mbb{L}} \mc{B}$ is natural on $\mc{B}$.


\end{document}